\documentclass[a4paper,11pt]{article}
\pdfoutput=1
\usepackage[left = 2.5cm, right = 2.5cm, top = 2.5cm, bottom = 2.5cm, headsep = 12pt, headheight = 15pt]{geometry}
\usepackage{parskip}
\usepackage{wrapfig}
\usepackage[T1]{fontenc}
\usepackage[utf8]{inputenc}
\usepackage{amsfonts,amsmath,amssymb,mathtools,microtype,anyfontsize}
\usepackage[noBBpl]{mathpazo} 
\usepackage[dvipsnames]{xcolor} 
\usepackage[hyphens]{url} 
\usepackage[linktoc = all, hidelinks]{hyperref} 
\hypersetup{colorlinks,linkcolor={blue!60!black},urlcolor={blue!60!black},citecolor={green!60!black}}
\usepackage{booktabs,caption,enumitem,float,graphicx,mdwlist,subcaption} 
\setlist[itemize]{topsep=0ex,itemsep=0ex,parsep=0.4ex}
\setlist[enumerate]{topsep=0ex,itemsep=0ex,parsep=0.4ex}
\usepackage{amsthm,thmtools,thm-restate} 
\usepackage[capitalise, compress, nameinlink, noabbrev]{cleveref} 
\declaretheorem[name = Theorem, numberwithin = section, style = plain]{theorem}

\declaretheorem[name = Lemma, numberlike = theorem, style = plain]{lemma}
\declaretheorem[name = Proposition, numberlike = theorem, style = plain]{proposition}
\usepackage[sort,compress,numbers]{natbib}
\renewcommand{\epsilon}{\varepsilon}

\renewcommand{\geq}{\geqslant}
\renewcommand{\leq}{\leqslant}

\DeclarePairedDelimiter{\abs}{\lvert}{\rvert} 

\DeclarePairedDelimiter{\set}{\lbrace}{\rbrace} 

\DeclareMathOperator{\im}{Im}

\newcommand{\defn}[1]{\textcolor{Maroon}{\emph{#1}}}
\renewcommand{\AA}{\mathcal{A}}
\newcommand{\BB}{\mathcal{B}}
\newcommand{\CC}{\mathcal{C}}

\newcommand{\RR}{\mathbb{R}}

\newcommand{\OO}{\mathcal{O}}

\newcommand{\eps}{\varepsilon}

\begin{document}

\title{\bf Non-Homotopic Drawings of Multigraphs}
\author{Ant\'onio Gir\~ao\footnotemark[1] \qquad Freddie Illingworth\footnotemark[2] \qquad Alex Scott\footnotemark[1] \qquad  David R.\ Wood\footnotemark[3]}

\maketitle

\begin{abstract}
    A multigraph drawn in the plane is \defn{non-homotopic} if no two edges connecting the same pair of vertices can be continuously deformed into each other without passing through a vertex, and is \defn{$k$-crossing} if every pair of edges (self-)intersects at most $k$ times. We prove that the number of edges in an $n$-vertex non-homotopic $k$-crossing multigraph is at most $6^{13 n (k + 1)}$, which is a substantial improvement over previous upper bounds.

    We also study this problem in the setting of \defn{monotone} drawings where every edge is an x-monotone curve. We show that the number of edges, $m$, in such a drawing is at most $2 \binom{2n}{k + 1}$ and the number of crossings is $\Omega\bigl(\frac{m^{2 + 1/k}}{n^{1 + 1/k}}\bigr)$. For fixed $k$ these bounds are both best possible up to a constant multiplicative factor. 
\end{abstract}

\renewcommand{\thefootnote}{\fnsymbol{footnote}} 

\footnotetext[1]{Mathematical Institute, University of Oxford, United Kingdom (\texttt{a.girao@ucl.ac.uk}, \texttt{alexander\allowbreak.scott\allowbreak @maths\allowbreak.ox\allowbreak.ac\allowbreak.uk}). Research supported by EPSRC grant EP/V007327/1.}

\footnotetext[2]{Department of Mathematics, University College London, United Kingdom (\texttt{f.illingworth@ucl.ac.uk}). Research supported by EPSRC grant EP/V521917/1 and the Heilbronn Institute for Mathematical Research.}

\footnotetext[3]{School of Mathematics, Monash University, Melbourne, Australia  (\texttt{david.wood@monash.edu}). Research supported by the Australian Research Council and a Visiting Research Fellowship of Merton College.}

\renewcommand{\thefootnote}{\arabic{footnote}} 

\section{Introduction}\label{sec:intro}

A central question of graph drawing research asks for the maximum number of edges in an $n$-vertex graph that admits a drawing satisfying certain properties. For example, it follows from Euler's formula that planar graphs have at most $3n-6$ edges (for $n \geq 3$), which is tight for planar triangulations. The situation for multigraphs (allowing parallel edges and self-loops) is less clear, since there are planar multigraphs with one vertex and arbitrarily many edges. To avoid this, the constraint that the drawing is \defn{non-homotopic} is typically applied: a drawing is non-homotopic if no pair of parallel edges or loops can be continuously deformed into each other without passing through vertices. The planar condition is also frequently relaxed so that the drawing is \defn{$k$-crossing}: the number of crossings between each pair of edges is at most $k$ and the number of self-intersections of each edge is at most $k$. Note that some bound on the number of crossings is required as there are non-homotopic drawings on a fixed number of vertices with arbitrarily many edges where different edges wind round the same vertex increasingly many times, as illustrated in \cref{ThreeVerticesManyEdges}.

\begin{figure}[ht]
\centering
\includegraphics{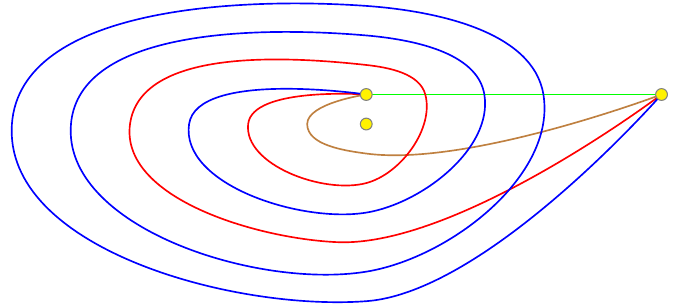}
\caption{Non-homotopic edges (with no self-crossings) winding around the same vertex increasingly many times.}
\label{ThreeVerticesManyEdges}
\end{figure}

In 1996, \citet*[Thm.~3.5]{JMM96} proved that $n$-vertex multigraphs with a non-homotopic $k$-crossing drawing have a bounded number of edges. Following this, the maximum cardinality of a collection of arcs that is non-homotopic and $k$-crossing on various surfaces has been well-studied. See, for example, the papers \cite{ABG19,AS18,BarNatan20,Douba22,FPS21,Greene19,MRT14,PTT22,Przytycki15,SP19}.

The first result of this paper is an improved upper bound on the number of edges in a non-homotopic $k$-crossing drawing of a multigraph.

\begin{theorem}
    \label{NonHomotopic}
    Let $n$, $k$ be positive integers. Let $G$ be an $n$-vertex multigraph that has a non-homotopic $k$-crossing drawing. Then
    \begin{equation*}
        \abs{E(G)} \leq 6^{13 n (k + 1)}.
    \end{equation*}
\end{theorem}

\Cref{NonHomotopic} improves on results from the literature, as we now explain. 
For integers $n, k \geq 1$, let $S$ denote the set obtained from $\mathbb{R}^2$ by removing $n$
distinct points, and fix $x\in S$. Let $f(n, k)$ be the maximum number of pairwise non-homotopic $x$-loops in $S$ such that none of them passes through $x$, each of them has fewer than $k$ self-intersections, and every pair of them cross fewer than $k$ times. 
The argument of \citet*{JMM96} shows that
\begin{equation}\label{eq:fUpperBound} 
   f(n, k) < (nk)^{\OO(nk^2)}.
\end{equation}
\Cref{NonHomotopic} (in the special case that every edge is a loop incident to a single vertex) improves the upper bound on $f(n, k)$ to
\begin{equation*}
    f(n, k) \leq 6^{13 n k}. 
\end{equation*}
Note that \citet*{PTT22} proved the exponential lower bound, 
\begin{equation}\label{eq:PTTLB}
    f(n, k) \geq 2^{\sqrt{nk}/3}.
\end{equation}
for $k \geq n \geq 2$.

Another common constraint applied to graph drawings is monotonicity: a drawing is \defn{monotone} if every edge is an x-monotone curve. See \citep{PT12,FPT10,PT04,PS11,FIKKMS16} for previous work on crossings in monotone drawings. The second direction of this paper is to obtain a bound on the number of edges in a multigraph with a monotone non-homotopic $k$-crossing drawing, and to prove an associated crossing lemma. We also derive the corresponding results when every pair of parallel edges cross at most $k$ times and every pair of incident edges cross at most $k$ times (a pair of edges is \defn{incident} if they share at least one common end-vertex). All of these results have optimal dependence on $n$.

\begin{theorem}\label{MonotoneNonHomotopic}
    Let $k < n$ be non-negative integers. For every $n$-vertex multigraph $G$ and every monotone non-homotopic drawing $D$\textup{:}
    \begin{enumerate}[label = \textup{(}\alph{*}\textup{)}]
        \item If each pair of edges cross at most $k$ times in $D$, then $\abs{E(G)} \leq 2 \cdot \binom{2 n}{k + 1}$. \label{edge:all}
        \item If each pair of incident edges cross at most $k$ times in $D$, then $\abs{E(G)} \leq n \cdot \binom{2 n}{k + 1}$. \label{edge:incident}
        \item If each pair of parallel edges cross at most $k$ times in $D$, then $\abs{E(G)} \leq n(n - 1) \cdot \binom{2 n}{k + 1}$. \label{edge:parallel}
    \end{enumerate}
\end{theorem}

The condition $k < n$ is fully general: if an $n$-vertex multigraph $G$ has a monotone drawing, then it has a homotopically equivalent monotone drawing in which every pair of edges cross fewer than $n$ times.\footnote{First project the monotone drawing to one that is homotopically equivalent where the vertices of $G$ lie at the points $(1, 0)$, $(2, 0)$, \ldots, $(n, 0)$ on the x-axis. For an edge $e$ in this drawing from $p_{x_1} = (x_1, 0)$ to $p_{x_2} = (x_2, 0)$, let $p_j$ ($j = x_1 + 1, \dotsc, x_2 - 1$) be the point where $e$ has x-coordinate $j$. Let $\tilde{e}$ be the piecewise linear curve going through $p_{x_1}, p_{x_1 + 1}, \dotsc, p_{x_2}$. The new drawing of $G$ where every edge $e$ is replaced by $\tilde{e}$ is homotopically equivalent to the original. Note that for any integer $a$, edges $\tilde{e}_1$ and $\tilde{e}_2$ are straight lines (or single points or not present) when restricted to the region $a \leq x \leq a + 1$ and so cross at most once in this region. Hence, every pair of edges cross at most $n - 1$ times.} In \cref{sec:MonotoneExample} we give examples that show the dependence upon $n$ is best possible in all three parts.

One motivation for such extremal theorems is that they are essential ingredients in proofs of crossing number lower bounds. The famous `Crossing Lemma' of \citet*{Ajtai82} and \citet{Leighton83} says that every drawing of any $n$-vertex simple graph (with no parallel edges and no loops) with $m \geq 4n$ edges has $\Omega\bigl(\frac{m^3}{n^2}\bigr)$ crossings. \Citet{Szekely97} proved an analogous result for multigraphs where the lower bound depends on the edge multiplicity. Recent research has shown that under certain assumptions, no dependence on the multiplicity is needed \citep{KPTU21,PT20,PTT22}. Similarly, using \cref{MonotoneNonHomotopic}, we prove the following three tight lower bounds on the number of crossings in monotone non-homotopic drawings where each pair of parallel edges/incident edges/edges cross at most $k$ times.

\begin{restatable}{theorem}{MonoNonHomCombined}\label{MonoNonHomCombinedCrossingLowerBound}
    For every non-negative integer $k$ there are positive constants $\alpha_k$, $\beta_k$, and $\gamma_k$ such that the following holds. For every $n$-vertex $m$-edge multigraph $G$ with $m \geq 4n$ and every monotone non-homotopic drawing $D$ of $G$\textup{:}
    \begin{enumerate}[label = \textup{(}\alph{*}\textup{)}]
        \item If each pair of edges cross at most $k$ times in $D$, then $D$ has at least $\alpha_k \frac{m^{2 + 1/k}}{n^{1 + 1/k}}$ crossings. \label{crossing:all}
        \item If each pair of incident edges cross at most $k$ times in $D$, then $D$ has at least $\beta_k \frac{m^{2 + 1/(k + 1)}}{n^{1 + 1/(k + 1)}}$ crossings. \label{crossing:incident}
        \item If each pair of parallel edges cross at most $k$ times in $D$, then $D$ has at least $\gamma_k \frac{m^{2 + 1/(k + 2)}}{n^{1 + 1/(k + 2)}}$ crossings. \label{crossing:parallel}
    \end{enumerate}
\end{restatable}

In \cref{sec:MonotoneExample} we give an example to show that the results in \cref{MonoNonHomCombinedCrossingLowerBound} are best possible up to the values of $\alpha_k$, $\beta_k$, and $\gamma_k$.


We now compare \cref{MonotoneNonHomotopic,MonoNonHomCombinedCrossingLowerBound} to related results in the literature. Here we conflate a vertex or edge with its image in a drawing.
\Citet*{PT20} and \citet*{KPTU21} defined a drawing of a multigraph to be:
\begin{itemize}
    \item \defn{separated} if any two parallel edges do not cross and the ``lens'' formed by their union has at least one vertex in its interior and at least one vertex in its exterior;
    \item \defn{single-crossing} if any pair of edges cross at most once;
    \item \defn{locally star-like} if no two edges with a common end-vertex cross.
\end{itemize}

Note that every separated drawing is non-homotopic but there are non-homotopic drawings that are not separated, since parallel edges might cross in a non-homotopic drawing. Also note that locally star-like is the same as each pair of incident edges cross at most $0$ times.

\Citet{PT20} defined a drawing to be \defn{branching} if it is separated, single-crossing, and locally star-like. They proved that the number of edges, $m$, in a branching drawing on $n$ vertices is at most $n(n - 2)$ and the number of crossings is $\Omega\bigl(\frac{m^3}{n^2}\bigr)$. \Citet*{KPTU21} proved that these results still hold without the single-crossing assumption. \Cref{MonotoneNonHomotopic}\ref{edge:incident} and \Cref{MonoNonHomCombinedCrossingLowerBound}\ref{crossing:incident} with $k = 0$ match these bounds (up to multiplicative constants), replacing the separated assumption by the weaker non-homotopic assumption and adding the monotonicity assumption. 
\Citet*{KPTU21} conjectured that every separated single-crossing (but not necessarily locally star-like) drawing on $n$ vertices has $\OO(n^2)$ edges. \Citet*{FPS21} proved a $\OO(n^2\log n)$ upper bound in this case. \Cref{MonotoneNonHomotopic}\ref{edge:all} in the $k = 1$ case gives a better bound on the number of edges without the assumption that parallel edges do not cross, and with the extra monotonicity assumption. 

The rest of the paper is structured as follows. We introduce formal notation in \cref{sec:notation}. In \cref{sec:nonhomotop} we prove \cref{NonHomotopic}. The examples showing the tightness of \cref{MonotoneNonHomotopic,MonoNonHomCombinedCrossingLowerBound} are given in \cref{sec:MonotoneExample}. We prove \cref{MonotoneNonHomotopic} in \cref{sec:mono} and \cref{MonoNonHomCombinedCrossingLowerBound} in \cref{sec:crossing}. We finish with open problems in \cref{sec:conc}.

We finish this introduction with a brief outline of the combinatorial idea appearing in the proof of \cref{NonHomotopic} (the proof of \cref{MonotoneNonHomotopic} is a much simpler version of this). Suppose $G$ is an $n$-vertex multigraph drawn in the plane. We will define a `frame' for $G$. This will be a small set of arcs (that are sub-curves of edges of $G$) and points in the plane such that if two edges of $G$ `interact' with the frame in the same way, then those edges are homotopic. We will record how an edge interacts with the frame using a signature: a list of the arcs of the frame that the edge crosses as well as some winding numbers around some points of the frame. We will record the signature in such a way that if two edges have the same signature, then they are homotopic. 
Thus, if the drawing is non-homotopic, then the number of edges is at most the number of possible signatures. Finally we will show that if the drawing is $k$-crossing, then the number of possible signatures may be bounded from above by a function of $n$ and $k$.

\paragraph{Note added in revision.} The referee of this paper pointed out the important work of \citet[Thm.~1.2 \& 1.5]{Przytycki15}, which 
predates references \cite{PT20,KPTU21,FPS21}, and 
implies the above-mentioned conjecture in \cite{KPTU21} as well as \cref{MonotoneNonHomotopic}  (with slightly less precise bounds). In fact, the result of \citet{Przytycki15} is significantly more general than \cref{MonotoneNonHomotopic} (replacing the monotone condition by a weaker assumption, and it works on any surface).
The topological proof of \citet{Przytycki15} is completely different to our combinatorial proof. Also, our construction of monotone drawings in \cref{sec:MonotoneExample} has a similar flavour to (but more edges than) a construction of \citet[Example~4.1]{Przytycki15}. 

\subsection{Notation}\label{sec:notation}

A \defn{drawing} of a multigraph $G$ is a function $\phi$ that maps each vertex $v\in V(G)$ to a point $\phi_v \in \RR^2$ and maps each edge $e = vw\in E(G)$ to a continuous curve $\phi_e \colon [0, 1] \to \RR^2$ with endpoints $\phi_v$ and $\phi_w$, such that:
\begin{itemize}
    \item every vertex is represented by a different point; that is, the points $(\phi_v \colon v \in V(G))$ are distinct;
    \item no edge passes through any vertex; that is, $\phi_e(t) \neq v$ for all $v \in V(G)$, $e \in E(G)$, and $t \in (0, 1)$;
    \item each pair of edges cross at a finite number of points, where a \defn{crossing} is an unordered pair of distinct pairs $(e, s), (f, t) \in E(G) \times (0, 1)$ with $\phi_e(s) = \phi_f(t)$.
\end{itemize}

Note that if $k$ edges meet at a point, then this contributes $\binom{k}{2}$ crossings. Parallel edges $e,f\in E(G)$ with end-vertices $v,w\in V(G)$ are \defn{homotopic} with respect to a drawing $\phi$ if there is a homotopy of $\mathbb{R}^2\setminus\set{\phi_x \colon x \in V(G) \setminus \set{v,w}}$ mapping $\phi_e$ to $\phi_f$ with $\phi_v$ and $\phi_w$ fixed. Here $\mathbb{R}^2 \setminus \set{\phi(x) \colon x\in V(G)\setminus\set{v,w}}$ is the plane with $\abs{V(G)} - 2$ puncture points. Intuitively, parallel edges $e$ and $f$ are homotopic if $e$ can be continuously deformed into $f$ without passing through any vertices and while keeping the end-points fixed. As illustrated in \cref{NonHomotopicParallelEdges}, a drawing of a multigraph is \defn{non-homotopic} if no two parallel edges are homotopic. It is natural to consider non-homotopic drawings, since homotopic edges can be considered to be equivalent. The non-homotopic assumption immediately eliminates the example of a planar multigraph with exactly one vertex and more than one edge.

\begin{figure}[ht]
	\centering
	\includegraphics{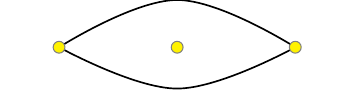}
	\caption{Non-homotopic parallel edges.}
	\label{NonHomotopicParallelEdges}
\end{figure}

A drawing $\phi$ of a multigraph $G$ is \defn{monotone} if $\phi_e$ is an x-monotone curve for each edge $e\in E(G)$; that is, $\phi_e$ intersects each vertical line in at most one point. Note that a monotone drawing has no self-loops and no edge self-intersects.
   
\section{Non-Homotopic Drawings}\label{sec:nonhomotop}

In this section we prove \cref{NonHomotopic}, which says that every non-homotopic $k$-crossing drawing on $n$ vertices has at most $6^{13n(k+1)}$ edges. This result follows quickly from the next lemma.

\begin{lemma}\label{NonHomotopicxy}
    Let $n, k$ be positive integers. Let $x, y$ be two \textup{(}not necessarily distinct\textup{)} vertices of an $n$-vertex multigraph $G$ such that every edge of $G$ is from $x$ to $y$ and $G$ has a non-homotopic $k$-crossing drawing. Then
    \begin{equation*}
       \abs{E(G)} \leq 6^{13 n (k + 1)}/n^2.
    \end{equation*}
\end{lemma}

\begin{proof}[Proof of \cref{NonHomotopic}]
    Let $G$ be an $n$-vertex multigraph as in the statement of \cref{NonHomotopic}. For $x, y \in V(G)$, let $G_{x, y}$ be the subgraph of $G$ consisting of the edges from $x$ to $y$. Then $G_{x, y}$ satisfies the hypothesis of \cref{NonHomotopicxy}, so $\abs{E(G_{x, y})} \leq 6^{13 n (k + 1)}/n^2$. The result follows since $E(G) = \bigcup_{x, y} E(G_{x, y})$.
\end{proof}

In fact, \cref{NonHomotopicxy} implies that \cref{NonHomotopic} holds even with the $k$-crossing assumption weakened to every pair of parallel edges cross at most $k$ times, every pair of loops incident to the same vertex cross at most $k$ times, and every edge has at most $k$ self-intersections.

Before proving \cref{NonHomotopicxy} we need a simple lemma about plane graphs as well as some basic facts about the winding number. We will think of each edge $e$ of a plane graph as having two \defn{sides}, labelled $e^+$ and $e^-$ arbitrarily. If $e$ is incident to two faces, then $e^+$ is incident to one of these faces and $e^-$ to the other. If $e$ is incident to only one face, then both $e^+$ and $e^-$ are both incident to this face. See \cref{fig:sides-of-edges}.

\begin{figure}[!h]
    \centering
    (a)
    \includegraphics[width=0.4\linewidth]{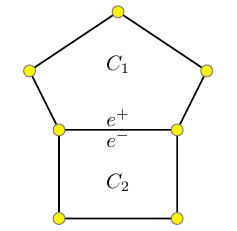}
    (b)     
    \includegraphics[width=0.4\linewidth]{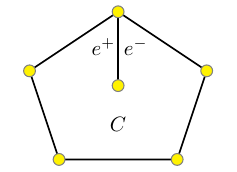}
    \caption{(a) Side $e^+$ is incident to face $C_1$, and $e^-$ is incident to $C_2$. (b) Sides $e^+$ and $e^-$ are incident to face $C$.        }
    \label{fig:sides-of-edges}
\end{figure}

\begin{lemma}\label{PlanarLongFaces}
    Let $G$ be a plane graph in which no vertices are isolated, all but at most $r$ vertices have degree at least three, and every face is incident to at least seven sides. Then $\abs{V(G)} \leq 10r - 28$.
\end{lemma}

\begin{proof}
    Let $n, m, F$ denote the number of vertices, edges and faces of $G$, respectively. Consider all pairs $(e^{\bullet}, f)$ where $f$ is a face of $G$ and $e^\bullet$ is a side of an edge of $G$ that is incident to $f$. Every edge $e$ corresponds to two such pairs (one for $e^+$ and one for $e^-$) and every face is in at least seven, so $7F \leq 2m$. Hence, by Euler's formula,
    \begin{equation*}
        n + \tfrac{2}{7} m \geq n + F \geq m + 2. 
    \end{equation*}
    Thus $m \leq \frac{7}{5} (n - 2)$. Finally,
    since $G$ has minimum degree at least one,
    \begin{equation*}
        \tfrac{14}{5} (n - 2) \geq 2m 
        = \sum_{v \in V(G)} \deg(v) 
        \geq 3(n - r) + r = 3n - 2r.
    \end{equation*}
    Rearranging gives the required result.
\end{proof}

The following facts about the winding number can be found in most texts on complex analysis or algebraic topology (such as~\citep{Fulton95}). Let $\gamma \colon [0, 1] \to \mathbb{C}$ be a continuous closed curve (so $\gamma(0) = \gamma(1)$). For any $a \notin \im \gamma$ there are continuous functions $r \colon [0, 1] \to \mathbb{R}_{>0}$ and $\theta \colon [0, 1] \to \mathbb{R}$ such that $\gamma = a + r e^{2 \pi i \theta}$. The \defn{winding number of $\gamma$ about $a$}, denoted by \defn{$I(\gamma, a)$}, is the integer $\theta(1) - \theta(0)$. This integer does not depend on the parameterisation $(r, \theta)$. It satisfies the following properties.
\begin{itemize}
    \item The winding number $I(\gamma, \cdot)$ is constant on each component of $\mathbb{C} \setminus \im \gamma$, is zero on the unbounded component and differs by one between adjacent components.
    \item If $\gamma$ is the union of two continuous closed curves $\gamma_1$ and $\gamma_2$, then $I(\gamma, a) = I(\gamma_1, a) + I(\gamma_2, a)$ for any $a \notin \im \gamma$.
\end{itemize}

The following two facts are also standard but do not always appear in textbooks. We include their proof for completeness.

\begin{lemma}\label{lem:winding}
    Let $a$ be a point in the plane and let $\gamma$, $\gamma_0$, and $\gamma_1$ be closed curves in $\RR^2 - a$.
    \begin{enumerate}
        \item $\gamma_0$ and $\gamma_1$ are homotopic in $\RR^2 - a$ if and only if they have the same winding number about $a$.
        \item $\gamma$ must self-intersect at least $\abs{I(\gamma, a)} - 1$ times \textup{(}counted with multiplicity\footnote{If $\gamma$ passes through the same point $k$ times, then this contributes $\binom{k}{2}$ self-intersection which is consistent with the definition of crossing in \cref{sec:notation}}\textup{)}.
    \end{enumerate}
\end{lemma}

\begin{proof}
    We first address part 1. Viewing $\gamma_0$ and $\gamma_1$ as continuous functions from $[0, 1]$ to $\mathbb{C}$, there are continuous functions $r_0, r_1 \colon [0, 1] \to \RR_{>0}$ and $\theta_0, \theta_1 \colon [0, 1] \to \RR$ such that $\gamma_j = a + r_j e^{2 \pi i \theta_j}$. First suppose that $\gamma_0$ and $\gamma_1$ have the same winding number $I$, so $I = \theta_1(1) - \theta_1(0) = \theta_0(1) - \theta_0(0)$. Let $H \colon [0, 1]^2 \to \mathbb{C} \setminus \set{a}$ be given by $a + r e^{2 \pi i \theta}$ where $r(\cdot, t) = t r_1 + (1 - t) r_0$ and $\theta(\cdot , t) = t \theta_1 + (1 - t) \theta_0$. Then $H$ is continuous, $H(\cdot, 0) = \gamma_0$, and $H(\cdot , 1) = \gamma_1$. Also, $\theta(1, t) = \theta(0, t) + I$ and so $H(\cdot, t)$ is a closed curve for each $t$. Thus $\gamma_0$ and $\gamma_1$ are homotopic.
    
    Next suppose that $\gamma_0$ and $\gamma_1$ are homotopic. Then there is a continuous function $H \colon [0, 1]^2 \to \mathbb{C} \setminus \set{a}$ where $H(\cdot, t)$ is a closed curve for each $t$, $H(\cdot, 0) = \gamma_0$, and $H(\cdot, 1) = \gamma_1$. Now $\tilde{\theta} \coloneqq \frac{1}{2 \pi }\arg(H - a) \colon [0, 1]^2 \to \RR/\mathbb{Z}$ is continuous and $\RR$ is a covering space for $\RR/\mathbb{Z}$ so, by the homotopy lifting property (see, for example, \citep[Prop~1.30]{Hatcher02}), there is a continuous $\theta \colon [0, 1]^2 \to \RR$ such that $\tilde{\theta}(x, t) \equiv \theta(x, t) \pmod{1}$ for all $x$ and $t$. The function $r \coloneqq \abs{H - a} \colon [0, 1]^2 \to \RR_{>0}$ is continuous. Thus we may write $H = a + r e^{2 \pi i \theta}$ for continuous $r \colon [0, 1]^2 \to \RR_{> 0}$ and $\theta \colon [0, 1]^2 \to \RR$. Since $H(\cdot, t)$ is a closed curve, $f(t) \coloneqq \theta(1, t) - \theta(0, t) \in \mathbb{Z}$ for all $t$. Note that $f$ is continuous and so constant. Thus $I(\gamma_1, a) = \theta(1, 1) - \theta(0, 1) = f(1) = f(0) = \theta(0, 1) - \theta(0, 0) = I(\gamma_0, a)$ and so $\gamma_0$ and $\gamma_1$ have the same winding number about $a$.
    
    We now address part 2. We will prove the following equivalent statement by induction on $\ell$: if $\gamma$ self-intersects $\ell$ times, then $\abs{I(\gamma, a)} \leq \ell + 1$. First suppose that $\ell = 0$. Then $\gamma$ is a simple closed curve. By the Jordan Curve Theorem, $\gamma$ splits the plane into two regions: an interior region bounded by $\gamma$ and an exterior unbounded region. $I(\gamma, \cdot)$ is zero on the unbounded region and is $\pm 1$ on the interior region (since the interior and exterior regions are adjacent). Now suppose that $\ell \geq 1$. Let $b$ be a point of self-intersection of $\gamma$. We may partition $\gamma$ into two closed curves $\gamma^{(0)}$ and $\gamma^{(1)}$ that meet at $b$. Let $\gamma^{(j)}$ have $\ell_j$ self-intersection. Then $\ell \geq \ell_0 + \ell_1 + 1$ (the $+1$ is for the self-intersection at $b$). Finally, by induction on $\ell$,
    \begin{equation*}
        \abs{I(\gamma, a)} = \abs{I(\gamma^{(0)}, a) + I(\gamma^{(1)}, a)} \leq \abs{I(\gamma^{(0)}, a)} + \abs{I(\gamma^{(1)}, a)} \leq (\ell_0 + 1) + (\ell_1 + 1) \leq \ell + 1. \qedhere
    \end{equation*}
\end{proof}

We are finally ready to prove \cref{NonHomotopicxy}.

\begin{proof}[Proof of \cref{NonHomotopicxy}]
We introduce some notation. Suppose that $H$ is a graph drawn in the plane and $E' \subseteq E(H)$ is a set of edges. We write $\CC(E')$ for the union of the arcs given by the edges of $E'$. This includes all end-points of edges of $E'$. For example, if $H$ is a $K_3$ with vertices $\set{(0,0), (1,0), (0,1)}$ and straight edges, then $\CC(E(H))$ is the perimeter of the right-angled triangle with vertices $\set{(0,0), (1,0), (0,1)}$.

We will first prove the result under the assumption that $x \neq y$ and at the end of the proof we will explain how the strategy is adapted to handle the $x = y$ case. Let the vertices of $G$ be $x, y, v_1, \dotsc, v_{n - 2}$.

Fix a non-homotopic drawing of $G$ in which every pair of edges cross at most $k$ times. Recall that all edges of $G$ are from $x$ to $y$. By applying a small perturbation we may and will assume that no three edges meet at a point. For $i = 1, \dotsc, n - 2$, let $e_i$ be an edge of $G$ passing closest to $v_i$ (the $e_i$ may not be distinct). Let $p_i$ be a point on $e_i$ closest to $v_i$ (the $p_i$ may not all be distinct, some of them may be crossings and some of them may be $x$ or $y$). By the minimality of the length of line-segment $p_i v_i$, no edge of $G$ crosses the interior of $p_i v_i$.

We are going to build a `frame' $F$ which we will use to encode the homotopy of edges of $G$: edges of $G$ that intersect $F$ `in the same way' will be homotopic. The frame $F$ will be a connected plane graph with $V(G) \subseteq V(F)$. We first build a frame $F_n$ that will be a plane tree containing all vertices of $G$. We use $f$ to refer to edges of a frame and $e$ for edges of $G$. 

As illustrated in \cref{BuildingFn}, start with $F_1$ which is the empty graph on vertex set $\set{x}$.  In step 1, let $f_1$ be a shortest sub-curve of $e_1 \cup p_1 v_1$ from $x$ to $v_1$. Note that $f_1$ is a non-self-intersecting curve from $x$ to $v_1$. Furthermore, all edges of $G - e_1$ intersect $f_1$ at most $k$ times. Let $F_2$ be the plane graph obtained from $F_1$ by adding the edge $f_1$ and the vertex $v_1$. In step $i$, for $i = 2, \dotsc, n - 2$, do the following:
\begin{enumerate}
    \item If $v_i \in V(F_i)$, then move straight onto step $i + 1$.
    \item Else if $v_i \in \CC(E(F_i))$, then let $F_{i + 1}$ be the plane tree obtained from $F_i$ by adding vertex $v_i$ (this subdivides an edge of $F_i$ into two). Move straight onto step $i + 1$.
    \item Else let $f_i$ be the shortest sub-curve of $e_i \cup p_i v_i$ that goes from $\CC(E(F_i))$ to $v_i$. Let $v'_i$ be the point where $f_i$ meet $\CC(E(F_i))$.
    \item If $v'_i \in V(F_i)$, then let $F_{i + 1}$ be the plane tree obtained from $F_i$ by adding edge $f_i$ and vertex $v_i$.
    \item If $v'_i \notin V(F_i)$, then let $F_{i + 1}$ be the plane tree obtained from $F_i$ by adding vertex $v'_i$ (this subdivides an edge of $F_i$ into two), adding edge $f_i$, and adding vertex $v_i$.
\end{enumerate}
This defines $F_1, \dotsc, F_{n - 1}$. If $y \in V(F_{n - 1})$, then take $F_n = F_{n - 1}$. Else if $y \in \CC(E(F_{n - 1}))$, then let $F_n$ be the plane graph obtained from $F_{n - 1}$ by adding vertex $y$ (this subdivides an edge of $F_{n - 1}$ into two). Else let $e_{n - 1}$ by any edge of $G$, and let $f_{n - 1}$ be the shortest sub-curve of $e_{n - 1}$ that goes from $\CC(E(F_{n - 1}))$ to $y$ (which is an end-point of $e_{n - 1}$). 
Let $v'_{n - 1}$ be the point where $f_{n - 1}$ meets $\CC(E(F_{n - 1}))$. If $v'_{n - 1} \in V(F_{n - 1})$, then let $F_n$ be the plane graph obtained from $F_{n - 1}$ by adding edge $f_{n - 1}$ and vertex $y$. If $v'_{n - 1} \notin V(F_{n - 1})$, then let $F_n$ be the plane graph obtained from $F_{n - 1}$ by adding vertex $v'_{n - 1}$ (this subdivides an edge of $F_{n - 1}$ into two), adding edge $f_{n - 1}$, and adding vertex $y$. 

\begin{figure}[!ht]
    \centering
    \includegraphics[width=\textwidth]{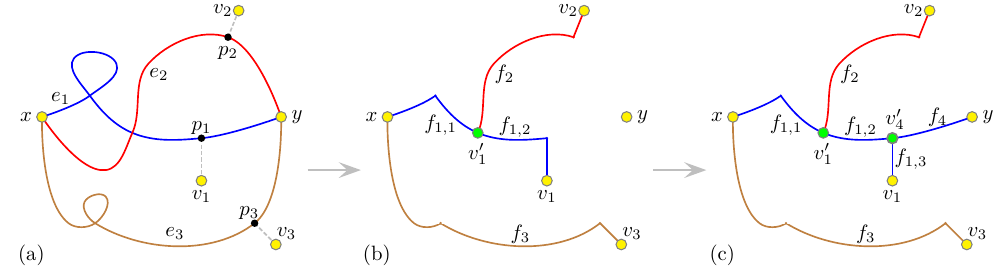}
    \caption{Example of the $n=5$ case: (a) drawing of $G$, (b) $F_{n-1}$ and (c) $F_n$, where $f_{1,1}$ $f_{1,2}$ and $f_{1,3}$ are obtained by subdividing $f_1$.}
    \label{BuildingFn}
\end{figure}

Note the following properties of $F_n$:
\begin{itemize}
    \item Since each step introduces at most two vertices, $F_n$ is a plane tree with at most $2n$ vertices.
    \item Since step $i$ introduces $v_i$ (if it is not already present), $V(G) \subseteq V(F)$.
    \item Since $v'_i$ has degree at least three when introduced, all vertices in $V(F) \setminus V(G)$ have degree at least three.
    \item $\CC(E(F_n))$ is a subset of the union of $\CC(\set{e_1, \dotsc, e_{n - 1}})$ and the line-segments $p_i v_i$ ($i = 1, \dotsc, n - 2$). In particular, each edge of $G - \set{e_1, \dotsc, e_{n - 1}}$ intersects $\CC(E(F_n))$ at most $k(n - 1)$ times.
\end{itemize}

We now construct a new frame $F$ (still a plane graph) building on $F_n$. Suppose, in the current frame, that sides $f^\circ$ ($\circ \in \set{+, -}$) and $f'^\bullet$ ($\bullet \in \set{+, -}$) are both incident to a face $C$. We say \defn{$f^\circ$ sees $f'^\bullet$ across $C$} if there is an edge $e$ of $G$ such that, at some point on $e$'s journey from $x$ to $y$, $e$ crosses edge $f$ leaving on side $f^\circ$ and then crosses edge $f'$ arriving on side $f'^\bullet$ and, in between, $e$ is in the interior of $C$ and crosses no other edge of $F$.\footnote{To be precise: using the $\phi$ notation defined in \cref{sec:notation}, there are $s < t$ in $[0, 1]$ such that $\phi_e(s)$ is in $f$, $\phi_e(t)$ is in $f'$, the curve $(\phi_e(s + \eps) \colon \eps \to 0^+)$ approaches $f$ on the side of $f^\circ$, the curve $(\phi_e(t - \eps) \colon \eps \to 0^+)$ approaches $f'$ on the side of $f'^\bullet$, and $\phi_e((s, t))$ is inside $C$ and disjoint from $\CC(E(F))$.}

We carry out the following process. Let $F$ be the current frame which is initially $F_n$. Let $E'$ be a subset of $E(G)$ initialised as $\set{e_1, \dotsc, e_{n - 1}}$. Iteratively do the following.
\begin{itemize}
    \item As illustrated in \cref{BuildingF}, suppose there is a side, $f^\circ$, of an edge and a face $C$ of $F$ such that $f^\circ$ sees at least nine other sides of edges of $F$ across $C$. Label nine of these sides $f_1^{\bullet_1}, f_2^{\bullet_2}, \dotsc, f_9^{\bullet_9}$ ($\bullet_1, \dots, \bullet_9 \in \set{+, -}$) in any order.
    \item Since $f^\circ$ sees $f_j^{\bullet_j}$ across $C$, there is an edge $e'_j$ of $G$ that crosses $f$ leaving on side $f^\circ$ and then crosses $f_j$ arriving on side $f_j^{\bullet_j}$ and, in between, is inside $C$ and crosses no other edges of $F$. Let $c'_j$ be the shortest sub-curve of $e'_j$ from $f$ to $f_j$ whose interior is inside $C$ and crosses no other edges of $F$.
    \item Since $f^\circ$ sees nine distinct sides, there is some choice of $j$ such that $c'_j$ splits $C$ into two faces both of which are incident to at least seven sides: a part of $f^\circ$, a part of $f_j^{\bullet_j}$, four of $\set{f_1^{\bullet_1}, \dotsc, f_9^{\bullet_9}} \setminus \set{f_j^{\bullet_j}}$, and one side of $c'_j$. Let $e$ be this $c'_j$ and $e'$ be $e'_j$.
    \item Add the edge $e$ to $F$ together with a vertex at the point where $e$ meets $f$ and a vertex at the point where $e$ meets $f_j$ (this subdivides each of $f$ and $f_j$ exactly once).
    \item Add $e'$ to $E'$.
\end{itemize}

\begin{figure}[!ht]
    \centering
    \includegraphics{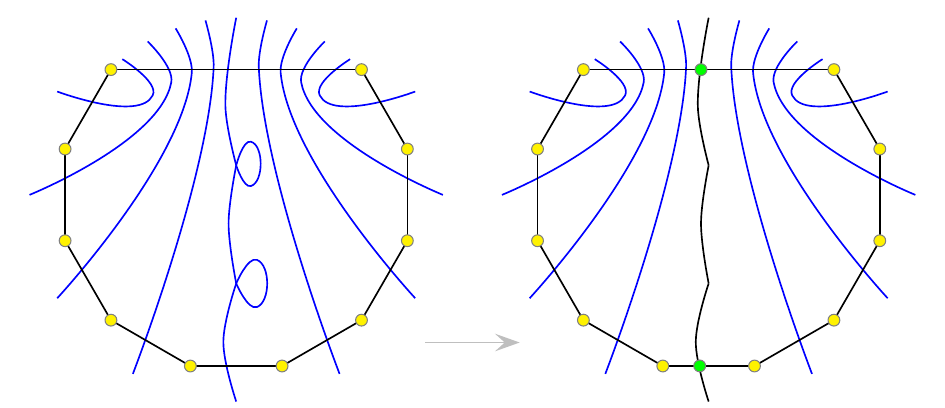}
    \caption{Building $F$. The edges of $F$ are black.}
    \label{BuildingF}
\end{figure}

The process stops when, for every side $f^\circ$ ($f \in E(F)$, $\circ \in \set{+, -}$) and every face $C$ of $F$ to which $f^\circ$ is incident, $f^\circ$ sees at most eight other sides across $C$. We now show the number of steps in this process is at most $5(n - 1)$ and the final frame has at most $10 n$ edges. First note that if $F_n$ has at most four edges, then no steps will be carried out and we have the required bounds. Otherwise, $F_n$ is a plane tree with at least five edges in which at most $n$ vertices have degree less than three. Every step introduces two vertices of degree three and splits a face into two faces both of which are incident to at least seven sides. Hence, every frame in the process is a connected plane graph, has at most $n$ vertices of degree less than 3 and every face of $F$ is incident to at least seven sides. Thus, by \cref{PlanarLongFaces}, every frame within the process has at most $10n - 28 \leq 10n$ vertices. Since each step adds two vertices, the process must stop within $5n - 14 \leq 5(n - 1)$ steps.

Note the following properties of the final frame $F$ and final set $E'$.
\begin{itemize}
    \item $F$ is a simple plane graph on at most $10n$ vertices. Each step introduces at most three new edges (the edge across the face plus one for each of the two subdivisions) and $F_n$ had at most $2n$ edges, so $F$ has at most $2n + 3 \cdot 5(n - 1) \leqslant 17 n$ edges.
    \item In $F$, every side of an edge sees at most nine sides across the face to which it is incident.
    \item Since the process stops within $5(n - 1)$ steps, $\abs{E'} \leq (n - 1) + 5(n - 1) = 6(n - 1)$.
    \item $\CC(E(F))$ is a subset of the union of $\CC(E')$ and the line-segments $p_i v_i$. In particular, each edge in $E(G) - E'$ intersects $\CC(E(F))$ at most $6(n - 1)k$ times.
\end{itemize}
We are now going to encode each edge of $G - E'$ using the sides of edges of $F$ so that if two edges of $G - E'$ had the same encoding, then they would be homotopic. Now, $F$ is a simple connected plane graph with one \defn{outer face} $C_{\infty}$. The rest of its faces are \defn{inner faces}. An edge $f$ of $F$ is an \defn{outer edge} if it is incident to the outer face of $F$. Otherwise $f$ is an \defn{inner edge}.

We will encode each edge $e \in E(G) - E'$ by writing down the sequence of sides of edges of $F$ that $e$ crosses (as it is traversed from $x$ to $y$) together with some further information. Suppose that $f_1^{\bullet_1}$ and $f_2^{\bullet_2}$ ($f_1, f_2 \in E(F)$, $\bullet_1, \bullet_2 \in \set{+, -}$) are two consecutive sides in the sequence for $e$: that is, $e$ arrives at $f_1$ on side $f_1^{\bullet_1}$, crosses $f_1$, arrives at $f_2$ on side $f_2^{\bullet_2}$, crosses $f_2$, and between $f_1$ and $f_2$ no other edges were crossed. If $e$ goes from $f_1$ to $f_2$ across an inner face $C$, then $f_1^{\bullet_1}, f_2^{\bullet_2}$ determine this portion of $e$ up to homotopy. However, this is not the case if $e$ goes from $f_1$ to $f_2$ across the outer face $C_{\infty}$; here we need to further specify how many times this portion of $e$ loops around $F$ and whether this looping is anticlockwise or clockwise.

To this end, for each pair of outer edges $f, f'$ of $F$, fix:
\begin{itemize}
    \item a simple curve $P_{f, f'} \subseteq \CC(E(F))$ that joins $f$ to $f'$ along the outer face $C_{\infty}$,
    \item a point $p_{f, f'}$ in the plane which is sufficiently close to (but not on) the curve $P_{f, f'}$ so that the perpendicular segment from $p_{f, f'}$ to $P_{f, f'}$ does not intersect any edge of $G$ or $F$.
\end{itemize}
If $x$ is incident to the outer face $C_{\infty}$, then, for each outer edge $f$ of $F$, fix a curve $P_{x, f} \subseteq \CC(E(F))$ that joins $x$ to $f$ along $C_{\infty}$ and fix a point $p_{x, f}$ in the plane that is sufficiently close to but not on $P_{x, f}$. 
If $y$ is incident to $C_{\infty}$, then  similarly define $P_{y, f}$ and $p_{y, f}$ for each outer edge $f$.

Now if an edge $e$ goes from $f_1$ to $f_2$ across $C_{\infty}$, then we encode this portion of the journey (call it $e'$) as follows. Form a closed curve consisting of $e'$, $P_{f_1, f_2}$ as well as parts of $f_1, f_2$ that join these up, orient this curve so that $e$ goes from $f_1$ to $f_2$, and define \defn{$I(e')$} to be the resulting winding number about $p_{f_1, f_2}$ (see \cref{fig:signature}). 

\begin{figure}[!h]
    \centering
    \includegraphics[width=0.6\linewidth]{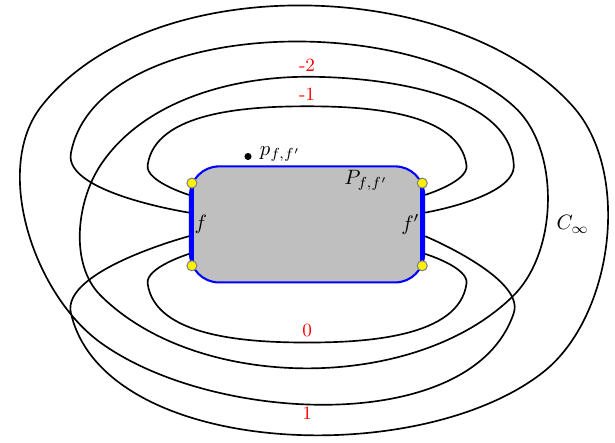}
    \caption{Signature}
    \label{fig:signature}
\end{figure}

By \cref{lem:winding}, for a point $p$ in the plane, two closed curves in $\RR^2 - p$ are homotopic if and only if they have the same winding number around $p$. Since the closed curve created here never enters an inner face of $F$, the winding number $I(e')$ determines $e'$ up to homotopy. Similarly, if an edge $e$ goes from $x$ to an outer edge $f$ across $C_{\infty}$ (call this portion $e'$) we form a closed curve consisting of $e'$, $P_{x, f}$ as well as the part of $f$ that joins these up and define \defn{$I(e')$} to be the winding number of this closed curve about $p_{x, f}$. Define $I(e')$ similarly if an edge $e$ goes from an outer edge $f$ to $y$.

We are now ready to define, for an edge $e \in E(G) - E'$, the \defn{signature of $e$}, denoted by \defn{$\sigma(e)$}. It is a sequence $a_1, f_1^{\bullet_1}, a_2, f_2^{\bullet_2}, \dotsc, a_{\ell}, f_{\ell}^{\bullet_{\ell}}, a_{\ell + 1}$ where:
\begin{itemize}
    \item The edge $e$ starts at $x$ and crosses edges $f_1, \dotsc, f_{\ell} \in E(F)$ in that order (each time arriving on the side $f_i^{\bullet_i}$) before arriving at $y$. The $f_i^{\bullet_i}$ need not be distinct.
    \item If $e$ goes from $f_{i - 1}$ to $f_{i}$ by crossing an internal face, then $a_i$ is defined to be $\ast$.
    \item If $e$ goes from $f_{i - 1}$ to $f_{i}$ by crossing $C_{\infty}$, then $a_i$ is the winding number $I(e')$ where $e'$ is this portion of $e$.
    \item If $e$ goes from $x$ to $f_1$ by crossing an internal face, then $a_1$ is defined to be $\ast$.
    \item If $e$ goes from $x$ to $f_1$ by crossing $C_{\infty}$, then $a_1$ is the winding number $I(e')$ where $e'$ is this portion of $e$.
    \item If $e$ goes from $f_\ell$ to $y$ by crossing an internal face, then $a_{\ell + 1}$ is defined to be $\ast$.
    \item If $e$ goes from $f_\ell$ to $y$ by crossing $C_{\infty}$, then $a_{\ell + 1}$ is the winding number $I(e')$ where $e'$ is this portion of $e$.
\end{itemize}
As discussed above, the signature of $e$ determines $e$ up to homotopy, so the number of edges in $E(G) - E'$ is at most the number of possible signatures. We look in some more detail at the value of $a_i$ when $e$ goes from $f_{i - 1}$ to $f_i$ by crossing $C_{\infty}$; denote this portion of $e$ by $e'$ so $a_i$ is the winding number $I(e')$. By \cref{lem:winding}, a closed curve with winding number $r$ about a point must self-intersect at least $\abs{r} - 1$ times. Hence, the closed curve consisting of $e', P_{f_{i - 1}, f_i}$ as well as the parts of $f_{i - 1}, f_i$ joining these up must self-intersect at least $\abs{a_i} - 1$ times. But none of these self-intersections can occur on $\CC(E(F))$ as the interior of $e'$ is within $C_{\infty}$. Thus $e'$ must self-intersect at least $\abs{a_i} - 1$ times. Similar reasoning applies for $a_1$ and $a_\ell$. Note the following properties of $\sigma(e)$:
\begin{itemize}
    \item Since each edge in $E(G) - E'$ intersects $\CC(E(F))$ at most $6(n - 1)k$ times, $\ell \leq 6(n - 1)k$.
    \item Since each edge self-intersects at most $k$ times, $\sum_i \max\set{\abs{a_i} - 1, 0} \leq k$ where we ignore summands corresponding to $\ast$.
    \item There are $2\abs{E(F)} \leq 34n$ choices for $f_1^{\bullet_1}$.
    \item Since every side sees at most nine sides across the face to which it is incident, there are, given $f_i^{\bullet_i}$, at most nine choices for $f_{i + 1}^{\bullet_{i + 1}}$.
\end{itemize}
We now bound from above the number of possible sequences satisfying these properties. Fix $\ell$. The number of possible sequences $f_1^{\bullet_1}, f_2^{\bullet_2}, \dotsc, f_{\ell}^{\bullet_\ell}$ in a signature is at most $34n \cdot 9^{\ell - 1}$. For each $i$, let $b_i = \max\set{\abs{a_i} - 1, 0}$ if $a_i \neq \ast$ and let $b_i = 0$ if $a_i = \ast$. The $b_i$ are $\ell + 1$ non-negative integers with sum at most $k$ and so there are at most $\binom{k + \ell + 1}{k}$ choices for $b_1, \dotsc, b_{\ell + 1}$. Given $b_i$ there are at most four options for $a_i$ (if $b_i = 0$, then $a_i$ could be $-1$, $0$, $1$, or $\ast$). Hence, the number of options for $a_1, \dotsc, a_{\ell + 1}$ is at most $4^{\ell} \binom{k + \ell + 1}{k}$. Thus, for fixed $\ell$, the number of signatures is at most
\begin{equation*}
    34n \cdot 9^{\ell - 1} \cdot 4^{\ell} \binom{k + \ell + 1}{k} \leq 4n \cdot 6^{2\ell} \binom{6nk}{k} \leq 4n \cdot 6^{12(n - 1)k} \cdot (6en)^{k}.
\end{equation*}
There are fewer than $6nk$ choices for $\ell$, and $\abs{E(G)}$ is at most the number of signatures plus $\abs{E'}$. So 
\begin{align*}
    \abs{E(G)} \leq 6nk \cdot 4n \cdot 6^{12(n - 1)k} \cdot (6en)^{k} + \abs{E'} & \leq 24 n^2 k \cdot (6en)^k \cdot 6^{12(n - 1)k} + 6(n - 1) \\
    & \leq 24 n^2 (k + 1) \cdot (6en)^k \cdot 6^{12(n - 1)k}.
\end{align*}
Hence, it suffices to prove that
\begin{equation*}
    24 n^4 (k + 1) \cdot (6en)^k \leq 6^{nk + 13n + 12k}.
\end{equation*}
But $n^4 \leq 6^n$, $24(k + 1) \leq 6^{3k}$, and $(6en)^k \leq (6^{n + 2})^k = 6^{nk + 2k}$, as required.

Now we outline how the above argument is altered to prove the result in the case $x = y$ (when all edges are $x$-loops). Let the vertices of $G$ be $x, v_1, \dotsc, v_{n - 1}$. We build the frame $F_n$ exactly as before but with $v_{n - 1}$ replacing $y$ when we build $F_n$ from $F_{n - 1}$. Again $F_n$ satisfies the following properties:
\begin{itemize}
    \item $F_n$ is a plane tree with at most $2n$ vertices.
    \item $V(G) \subseteq V(F_n)$ and all vertices in $V(F_n) \setminus V(G)$ have degree at least three.
    \item Each edge of $G - \set{e_1, \dotsc, e_{n - 1}}$ intersects $\CC(E(F_n))$ at most $k(n - 1)$ times.
\end{itemize}
We construct the frame $F$ from $F_n$ exactly as before. The process stops when, in $F$, every side sees at most nine sides across the face to which it is incident. Again, by \cref{PlanarLongFaces}, the number of steps in this process is at most $5(n - 1)$ and the final frame has at most $10n$ vertices. The final frame $F$ and final set of edges $E'$ satisfy the same properties as before.

We now encode each edge $e \in E(G) - E'$. Direct each such edge and encode them as before by writing down the signature $a_1, f_1^{\bullet_1}, \dotsc, a_{\ell}, f_{\ell}^{\bullet_{\ell}}, a_{\ell + 1}$ where the $f_i$ are the sequence of edges of $F$ that $e$ crosses and the $a_i$ are either winding numbers (if $e$ is crossing the unbounded face $C_{\infty}$) or are $\ast$ (if $e$ is crossing an internal face). The only difference now is that there is no need for the curves $P_{y, f}$ and points $p_{y, f}$.

We then finish the argument as before.
\end{proof}

\Cref{NonHomotopic} applies for $k \geq 1$. The $k = 0$ case (with no crossings) is handled by the following elementary result.

\begin{proposition}\label{NonHomotopicPlaneMultigraph}
    Every non-homotopic plane multigraph on $n \geq 2$ vertices has at most $4n - 4$ edges.
\end{proposition}

\begin{proof}
    Let $G$ be a non-homotopic plane multigraph on $n \geq 2$ vertices with $m$ edges. By adding edges, we may assume that $G$ is connected. 

A loop of $G$ is  \defn{trivial} if it has no vertices in its interior. Any two trivial loops incident to the same vertex are homotopic, so $G$ has at most $n$ trivial loops. Let $G'$ be obtained from $G$ by deleting all 
the trivial loops, and let $m' \coloneqq \abs{E(G')}$. Then $m \leqslant m' + n$. We claim that $m' \leq 3n - 4$ which will suffice. Since $G$ is connected, so too is $G'$. Define the \defn{length} of a face $F$ to be the number of edges (counting multiplicity) in a walk around $F$. As an example, the length of the face $F$ shown in \cref{eF} is 3 (the edge $e$ is traversed twice, once along each side, in a walk around $F$).

\begin{figure}[!htb]
	\centering
	\includegraphics{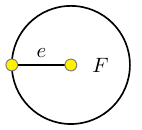}
	\caption{A face $F$ of length $3$.}
	\label{eF}
\end{figure}

Consider a face $F$ of length at most 2. If $F$ has a vertex in its interior, then, since $G'$ is connected, there is an edge from an interior vertex to a vertex on the boundary of $F$ and so $F$ has length at least 3. Thus $F$ has no vertex in its interior. The boundary of $F$ must either be a loop or a pair of parallel edges between two vertices $u$ and $v$. Since $G'$ contains no trivial loops and parallel edges are not homotopic, $F$ must be the outer face of $G'$. In particular, every face of $G'$ except the outer face has length at least 3. Consider the sum of the lengths of the faces of $G'$. Each edge contributes 2 to this sum and so the sum is equal to $2m'$. On the other hand, each face, except the outer face, contributes 3. Since $n \geq 2$ and $G'$ is connected, the outer face contributes at least 1. In particular, the sum is at least $3f - 2$ where $f$ is the number of faces of $G'$. Thus $3f - 2 \leq 2m'$. By Euler's formula, $2 + m' - n \leq f \leq \frac{2}{3}(m' + 1)$, and so $m' \leq 3n - 4$, as desired.
\end{proof}

The bound in \cref{NonHomotopicPlaneMultigraph} is tight. For $n = 2$, take the multigraph on $\set{u, v}$ consisting of edge $uv$, trivial loops at $u$ and $v$ as well as a loop at $u$ that has all the other edges in its interior. For $n \geq 3$, start with any simple plane triangulation on $n$ vertices, with the outer face bounded by the triangle $(u, v, w)$. Add a parallel edge $uv$ so that $w$ is no longer on the outer face boundary. Add a loop incident to $u$ so that $v$ is no longer on the outer face boundary. Finally, for each vertex $x$ add a loop incident to $x$ drawn in an internal face incident to $x$. The resulting drawing is non-homotopic, and with $(3n - 6) + 1 + 1 + n = 4n - 4$ edges.

\section{Construction of Monotone Drawings}\label{sec:MonotoneExample}

Here we give a construction showing the tightness of \cref{MonotoneNonHomotopic,MonoNonHomCombinedCrossingLowerBound}.
For integers $1 \leq k < n$ define the multigraph $G_{n, k}$ as follows, and illustrated in \cref{MonotoneExample}. It has vertex set $V(G_{n, k}) = \set{v_1,\dots,v_n}$ where each vertex $v_i$ is drawn at $(i, 0)$ in the plane. For integers $1 \leq a_0 < a_1< \dotsb < a_{\ell} \leq n$ where $1 \leq \ell \leq k$, add to $G_{n, k}$ an edge $e = e_{a_0, a_1, \dotsc, a_{\ell}}$ with endpoints $v_{a_0}$ and $v_{a_{\ell}}$, drawn as a degree-$(k + 1)$ monic polynomial
\begin{equation*}
    y = P_e(x) \coloneqq x^{k - \ell} (x - a_0)(x - a_1 - \tfrac12)(x - a_2 - \tfrac12) \dotsm (x - a_{\ell - 1} - \tfrac12)(x - a_{\ell})
\end{equation*}
with domain $a_0 \leq x \leq a_{\ell}$. This is a monotone curve crossing the x-axis at $(a_i + \frac{1}{2},0)$ for each $i \in \set{1, \dots, \ell - 1}$. The resulting multigraph $G_{n, k}$ has $\sum_{\ell = 1}^k \binom{n}{\ell + 1} = \Theta_k(n^{k + 1})$ edges.

\begin{figure}[ht]
    \centering
    \includegraphics{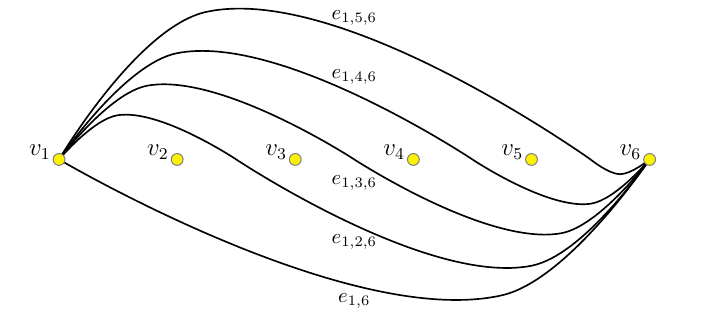}
    \caption{Edges between $v_1$ and $v_6$ in $G_{n, 2}$.}
    \label{MonotoneExample}
\end{figure}

Suppose that two edges $e = e_{s, a_1, \dotsc, a_{\ell - 1}, t}$ and $f = e_{s, b_1, \dotsc, b_{\ell' - 1}, t}$ between vertices $v_s$ and $v_t$ are homotopic. Then the signs of $P_e$ and $P_f$ must agree at each integer $i$ strictly between $s$ and $t$. Every crossing of one of the polynomials with the x-axis occurs at $x = m + \frac{1}{2}$ for $m \in \set{s + 1, \dotsc, t - 1}$ and all such crossings except $t - \frac{1}{2}$ correspond to a change in sign of the polynomial. Hence $P_e$ and $P_f$ must cross the x-axis at the same places (except possibly for $t - \frac{1}{2}$). In particular, the sets $\set{a_1, \dotsc, a_{\ell - 1}}$ and $\set{b_1, \dotsc, b_{\ell' - 1}}$ are the same or differ by the element $t - 1$. If they are the same, then $e = f$. If they differ by the element $t - 1$, then WLOG $\ell' = \ell - 1$, $a_1 = b_1$, \ldots, $a_{\ell - 2} = b_{\ell - 2}$, and, $a_{\ell - 1} = t - 1$. But then $xP_e(x) = (x - t + \frac{1}{2}) P_f(x)$ and so $P_e$ and $P_f$ differ in sign at $x = t - 1$ and so $e$ and $f$ are not homotopic. Therefore, the drawing of $G_{n, k}$ given is indeed non-homotopic.

Consider two distinct edges $e$ and $f$ of $G$. Both $P_e$ and $P_f$ are monic polynomials of degree $k + 1$, so their difference is a polynomial of degree at most $k$, which has at most $k$ roots. Thus the edges $e$ and $f$ cross at most $k$ times. Now suppose that edges $e$ and $f$ are incident at vertex $v_a$. Then both $e$ and $f$ are monic polynomials of degree $k + 1$ with a root at $a$, so their difference is a polynomial of degree at most $k$ with a root at $a$. The root at $a$ is not a crossing point of $e$ and $f$, so $e$ and $f$ cross at most $k - 1$ times. Similarly if $e$ and $f$ are parallel, then they cross at most $k - 2$ times. In particular,
\begin{itemize}
    \item $G_{n, k}$ has a monotone non-homotopic drawing in which each pair of edges cross at most $k$ times,
    \item $G_{n, k + 1}$ has a monotone non-homotopic drawing in which each pair of incident edges cross at most $k$ times,
    \item $G_{n, k + 2}$ has a monotone non-homotopic drawing in which each pair of parallel edges cross at most $k$ times.
\end{itemize}
Recall that $G_{n, k}$ has $\sum_{\ell = 1}^k \binom{n}{\ell + 1} \geq \binom{n}{k + 1}$ edges while the upper bounds from \cref{MonotoneNonHomotopic} parts \ref{edge:all}, \ref{edge:incident}, and \ref{edge:parallel} are $2 \cdot \binom{2n}{k + 1}$, $n \cdot \binom{2n}{k + 1}$, and $n (n - 1) \cdot \binom{2n}{k + 1}$, respectively.
This shows that the dependence on $n$ is best possible in \cref{MonotoneNonHomotopic} except for the case $k = 0$ in part~\ref{edge:all}. For this particular case we note that any straight-line plane triangulation gives a monotone non-homotopic drawing with $3n - 6$ edges and no crossings.

Next we address the tightness of \cref{MonoNonHomCombinedCrossingLowerBound}. Let $t \in \set{2k, \dotsc, n}$. Let $G_{n, k, t}$ be the sub-multigraph of $G_{n, k}$ consisting of those edges $e_{a_0, \dotsc, a_{\ell}}$ with $a_{\ell} - a_0 \leq t$. Each vertex has degree between $\sum_{\ell = 1}^k \binom{t}{\ell}$ and $2 \sum_{\ell = 1}^k \binom{t}{\ell}$, that is, has degree $\Theta_k(t^k)$. Thus $G_{n, k, t}$ has $m = \Theta_k(n t^k)$ edges.

Fix an edge $e = e_{a_0, \dotsc, a_{\ell}}$. Suppose an edge $f = e_{b_0, \dotsc, b_{\ell}}$ crosses $e$. Then $b_0 < a_{\ell}$ and $b_{\ell} > a_0$. Since $a_{\ell} - a_0 \leq t$ and $b_{\ell} - b_0 \leq t$, we have $b_0 \geq b_{\ell} - t > a_0 - t$ and $b_0 < a_{\ell} \leq a_0 + t$. In particular, there are at most $2t$ choices for $b_0$. Each vertex has degree $\Theta_k(t^k)$, so $e$ crosses at most $\OO_k(t^{k + 1})$ other edges. Hence the number of crossings in $G_{n, k, t}$ is at most $m \cdot \OO_k(t^{k + 1}) = \OO_k(n t^{2k + 1})$. But $n t^{2k + 1} = (n t^k)^{2 + 1/k}/n^{1 + 1/k} = \Theta_k(m^{2 + 1/k}/n^{1 + 1/k})$, so $G_{n, k, t}$ has $\OO_k(m^{2 + 1/k}/n^{1 + 1/k})$ crossings.

Thus $G_{n, k, t}$ shows \cref{MonoNonHomCombinedCrossingLowerBound}\ref{crossing:all} is tight where $m$ can be chosen independently of $n$ by varying $t$. $G_{n, k + 1, t}$ and $G_{n, k + 2, t}$ show that parts \ref{crossing:incident} and \ref{crossing:parallel} are tight, respectively.

\section{Number of Edges in Monotone Drawings}\label{sec:mono}

In this section we prove \cref{MonotoneNonHomotopic}. The main work is part \ref{edge:all}; the other two parts then follow quickly.

\begin{proof}[Proof of \cref{MonotoneNonHomotopic}\ref{edge:incident} and \ref{edge:parallel} assuming \cref{MonotoneNonHomotopic}\ref{edge:all}]
    Let $G$ be an $n$-vertex multigraph with a monotone non-homotopic drawing in which each pair of incident edges cross at most $k$ times. For each vertex $x$, let $G_x$ denote the sub-multigraph of $G$ consisting of those edges incident to $x$. Then $G_x$ is an $n$-vertex multigraph with a monotone non-homotopic drawing in which each pair of edges cross at most $k$ times. By \cref{MonotoneNonHomotopic}\ref{edge:all}, $\abs{E(G_x)} \leq 2 \cdot \binom{2n}{k + 1}$, implying 
    \begin{equation*}
        \abs{E(G)} = \tfrac{1}{2} \sum_{x \in V(G)} \abs{E(G_x)} \leq n \cdot \binom{2n}{k + 1},
    \end{equation*}
    which proves \cref{MonotoneNonHomotopic}\ref{edge:incident}.

    Now let $G$ be an $n$-vertex multigraph with a monotone non-homotopic drawing in which each pair of parallel edges cross at most $k$ times. For each pair of vertices $\set{x, y} \in \binom{V(G)}{2}$, let $G_{x, y}$ denote the sub-multigraph of $G$ consisting of those edges between $x$ and $y$. By \cref{MonotoneNonHomotopic}\ref{edge:all}, $\abs{E(G_{x, y})} \leq 2 \cdot \binom{2n}{k + 1}$. Thus,
    \begin{equation*}
        \abs{E(G)} = \sum_{\set{x, y} \in \binom{V(G)}{2}} \abs{E(G_{x, y})} \leq n (n - 1) \cdot \binom{2n}{k + 1},
    \end{equation*}
    which proves \cref{MonotoneNonHomotopic}\ref{edge:parallel}.
\end{proof}

It remains to prove \cref{MonotoneNonHomotopic}\ref{edge:all}. To do so we first encode the edges in the monotone non-homotopic $k$-crossing drawing of $G$ by a sequence of length $n$ and then bound the number of such sequences. We consider sequences in which each entry is in $\set{+,0,-,\ast}$. Such a sequence is a \defn{monotone drawing sequence} if it is of one of the two following forms:
\begin{align*}
    & \underbrace{\ast, \dotsc, \ast}_{\geq 0}, 0, \underbrace{\pm, \dotsc, \pm}_{\geq 0}, 0, \underbrace{\ast, \dotsc, \ast}_{\geq 0} \\[4pt]
    & \underbrace{\pm, \dotsc, \pm}_{\geq 0}, 0, \underbrace{\ast, \dotsc, \ast}_{\geq 0}
\end{align*}
where $\pm$ denotes that the entry is either a $+$ or a $-$ (independently between different symbols). In other words, there are two types of sequences: those with two 0's which have a sequence of $+$ and $-$ symbols between them and $\ast$'s at both ends, and those with one 0 which has a sequence of $+$ and $-$ symbols before and a sequence of $\ast$'s after. We give some intuition for these two sequence types. The first type encodes the edges in a monotone drawing: the 0s correspond to the end-vertices of the edges and the $\pm$ to whether the edge is above or below the vertices in between. The second type is needed for bounding the number of monotone drawing sequences: we will do so by induction, removing the first entry which can convert a sequence of the first type into one of the second.

Partially order the four symbols, $- < 0 < +$, where $\ast$ is incomparable with the others. Given two monotone drawing sequences $(a_i)$ and $(b_i)$ and an index $j$ we say that \defn{$(a_i)$ is below $(b_i)$ at $j$} if $a_j < b_j$ and \defn{$(a_i)$ is above $(b_i)$ at $j$} if $a_j > b_j$. Note that if $a_j$ or $b_j$ is a $\ast$, then the sequences are incomparable or equal at $j$.

Two sequences $(a_i)$ and $(b_i)$ \defn{cross $k$ times} if there are indices $i_0 < i_1 < \dotsb < i_k$ such that $(a_i)$ is below $(b_i)$ at $i_0, i_2, i_4, \dotsc$ and $(a_i)$ is above $(b_i)$ at $i_1, i_3, i_5, \dotsc$ or vice versa. Note that if $(a_i)$ and $(b_i)$ have length $n$, then they cross at most $n - 1$ times.

\begin{lemma}
\label{drawingtosequences}
If an $n$-vertex multigraph $G$ has a monotone drawing in the plane such that no edges are homotopic and every pair of edges cross at most $k$ times, then there exists a set of $\abs{E(G)}$ monotone drawing sequences of length $n$ any two of which cross at most $k$ times. 
\end{lemma}

\begin{proof}
We may perturb the vertices so that no two vertices have the same x-coordinate. Let $v_1, \dotsc, v_n$ be the vertices of $G$ ordered by increasing x-coordinate. Let $L_i$ be the vertical line through $v_i$. Consider each edge $e$ with endpoints $v_i$ and $v_j$ with $i < j$. For $\ell \in \set{1,\dotsc,i - 1} \cup \set{j + 1,\dotsc, n}$, let $s_{e, \ell}= *$. Let $s_{e, i} = 0$ and $s_{e, j} = 0$. 
For  $\ell \in \set{i + 1, \dots,j - 1}$, let $s_{e, \ell} = +$ if $e$ crosses $L_\ell$ above $v_\ell$, and let $s_{e,\ell}= -$ if $e$ crosses $L_\ell$ below $v_\ell$. Since $e$ is x-monotone, $s_{e, \ell}$ is well-defined. Let $s_e \coloneqq (s_{e, 1}, \dotsc, s_{e, n})$. By construction, $s_e$ is a monotone drawing sequence. Fix distinct edges $e, f \in E(G)$. Since $e$ and $f$ are not homotopic, $s_e \neq s_f$. Also, if there are indices $i < j$ with $s_e$ above $s_f$ at $i$ and $s_e$ below $s_f$ at $j$, then $e$ and $f$ must have a crossing whose x-coordinate is strictly between $i$ and $j$. Hence, $s_e$ and $s_f$ cross at most $k$ times.
\end{proof}

\begin{lemma}\label{sequencebound}
    Let $k < n$ be positive integers. Let $\AA$ be a set of different monotone drawing sequences of length $n$, each pair of which cross at most $k$ times. Then
    \begin{equation*}
        \abs{\AA} \leq 2 \cdot \binom{2n}{k + 1}.
    \end{equation*}
\end{lemma}

\begin{proof}
    Let $g(n, k)$ denote the size of the largest set of length $n$ monotone drawing sequences where every pair crosses at most $k$ times. In this definition we allow $k \geq n$, in which case $g(n, k) = g(n, n - 1)$, since no two monotone drawing sequences of length $n$ cross $n$ times. We also allow $k = -1$, in which case $g(n, -1) = 1$ (as no two sequences cross at most $-1$ times).

    Let $\AA$ be a set of monotone drawing sequences of length $n + 1$, such that every pair crosses at most $k$ times. The proof will proceed by deleting the first entry of each sequence in $\AA$ and analysing how many sequences elide.

    Let $\BB$ be those sequences of $\AA$ which, upon deleting their first entry, are no longer monotone drawing sequences. The only sequence in $\BB$ is $0, \ast, \ast, \dots, \ast$ and so $\abs{\BB} \leq 1$. Let $\AA_{+}, \AA_{0}, \AA_{-}, \AA_{\ast}$ be those sequences in $\AA - \BB$ which start with a $+$, $0$, $-$, $\ast$, respectively. Let $\AA'_{+}$ be those sequences obtained by deleting the first entry (which, of course, is a $+$) from sequences in $\AA_{+}$. Similarly define $\AA'_{0}, \AA'_{-}, \AA'_{\ast}$. Note that each of these is a set of monotone drawing sequences of length $n$. These sets may not be disjoint. Let $\AA' = \AA'_{+} \cup \AA'_{0} \cup \AA'_{-} \cup \AA'_{\ast}$ be those sequences obtained by deleting the first entry from a sequence in $\AA - \BB$. Note $\AA'$ is a set of monotone drawing sequences of length $n$ where every pair of sequences crosses at most $k$ times and so $\abs{\AA'} \leq g(n, k)$.

    In any monotone drawing sequence, the symbol $\ast$ can be only be succeeded or preceded by a $0$ or $\ast$. In particular, the first entry of a sequence in $\AA'_{\ast}$ is a $\ast$ or a $0$. If a sequence in $\AA'_{+} \cup \AA'_{0} \cup \AA'_{-}$ starts with a $\ast$, then the original sequence (before deleting the first entry) must have been $0, \ast, \ast, \dotsc, \ast$ and this was discounted when we removed $\BB$. Further, the only sequence in $\AA'_{+} \cup \AA'_{0} \cup \AA'_{-}$ starting with a $0$ is $0, \ast, \ast, \dotsc, \ast$. Hence, $\abs{\AA'_{\ast} \cap (\AA'_{+} \cup \AA'_{0} \cup \AA'_{-})} \leq 1$.

    We now consider those sequences that are in $\AA'_{+} \cap (\AA'_{0} \cup \AA'_{-})$ and show that every pair of them crosses at most $k - 1$ times. Indeed, suppose that $(a_i), (b_i) \in \AA'_{+} \cap (\AA'_{0} \cup \AA'_{-})$ cross $k$ times. Let $j$ be the first index where $a_j \neq b_j$.  We first show that neither $a_j$ nor $b_j$ is a $\ast$. Suppose that $a_j = \ast$. Since $(a_i) \in \AA'_{+} \cup \AA'_{0} \cup \AA'_{-}$, its first entry is a $+$, $0$, or $-$. Hence, from the definition of monotone drawing sequences, there is some $\ell < j$ with $a_\ell = 0$. By the minimality of $j$, $b_\ell$ is also $0$. But $(b_i) \in \AA'_{+} \cup \AA'_{0} \cup \AA'_{-}$, so every entry after $b_\ell$ is a $\ast$ and so $b_j = \ast$, contradicting $a_j \neq b_j$. 

    Since $a_j$, $b_j$ are distinct and neither is $\ast$, $(a_i)$ is either above or below $(b_i)$ at $j$. Without loss of generality $(a_i)$ is below $(b_i)$ at $j$. But $(a_i)$ and $(b_i)$ agree before $j$ and cross $k$ times, so there are indices $j = i_1 < i_2 < \dotsb < i_{k + 1}$ with $(a_i)$ below $(b_i)$ at $i_1, i_3, \dotsc$ and $(a_i)$ above $(b_i)$ at $i_2, i_4, \dotsc$. As $(a_i) \in \AA'_{+}$, the sequence obtained by appending a $+$ to the start of $(a_i)$ is in $\AA$, while, since $(b_i) \in \AA'_{0} \cup \AA'_{-}$, the sequence obtained by appending a $0$ or $-$ (depending on whether $(b_i) \in \AA'_{0}$ or $(b_i) \in \AA'_{-}$) to the start of $(b_i)$ is in $\AA$. These two sequences cross $k + 1$ times, a contradiction.

    Hence, $\AA'_{+} \cap (\AA'_{0} \cup \AA'_{-})$ is a set of monotone drawing sequences of length $n$, every pair of which cross at most $k - 1$ times and so $\abs{\AA'_{+} \cap (\AA'_{0} \cup \AA'_{-})} \leq g(n, k - 1)$. Similarly, $\abs{\AA'_{-} \cap (\AA'_{0} \cup \AA'_{+})} \leq g(n, k - 1)$. Hence, the number of sequences in at least two of $\AA'_{+}, \AA'_{0}, \AA'_{-}, \AA'_{\ast}$ is at most $2 g(n, k - 1) + 1$. By the inclusion-exclusion principle we have
    \begin{align*}
       \abs{\AA}  
       = \abs{\AA'_{+}} + \abs{\AA'_{0}} + \abs{\AA'_{-}} + \abs{\AA'_{\ast}} + \abs{\BB} 
        & \leq \abs{\AA'} + 2g(n, k - 1) + 2 \\
        & \leq g(n, k) + 2g(n, k - 1) + 2.
    \end{align*}
    Hence, $g(n + 1, k) \leq g(n, k) + 2g(n, k - 1) + 2$. It suffices to show that $g(n, k) \leq 2 \cdot \binom{2n}{k + 1}$ for all $-1 \leq k \leq n - 1$. There are some easy base cases:
    \begin{itemize}
        \item $g(n, -1) = 1$ for all $n$ (any two sequences cross at least 0 times);
        \item $g(1, k) = 1$ for all $k \geq 0$ (the only monotone drawing sequence of length 1 is $0$);
        \item $g(2, k) = 4$ for all $k \geq 0$ (the only monotone drawing sequences of length 2 are $+, 0$; $0,0$; $-, 0$; and $0,\ast$).
    \end{itemize}
    These and the recurrence give $g(n, 0) \leq 4n - 4$ for $n \geq 2$. Hence, the required inequality holds whenever $n \leq 2$ or $k \leq 0$. Suppose the result holds for $n \geq 2$. First suppose that $1 \leq k \leq n - 1$. Then
    \begin{align*}
        g(n + 1, k) & \leq g(n, k) + 2g(n, k - 1) + 2 \leq 2 \tbinom{2n}{k + 1} + 4 \tbinom{2n}{k} + 2 \\
        & \leq 2 \tbinom{2n}{k + 1} + 4 \tbinom{2n}{k} + 2 \tbinom{2n}{k - 1} = 2 \tbinom{2n + 2}{k + 1}.
    \end{align*}
    Finally suppose that $k = n$. Now
    \begin{align*}
        g(n + 1, n) & \leq g(n, n) + 2g(n, n - 1) + 2 = 3g(n, n - 1) + 2 \\
        & \leq 6 \tbinom{2n}{n} + 2 \leq 2 \tbinom{2n + 2}{n + 1},
    \end{align*}
    where the last inequality holds as $\binom{2n + 2}{n + 1} \binom{2n}{n}^{-1} = (4n + 2)/(n + 1) > 3$ for $n \geq 2$.
\end{proof}

\begin{proof}[Proof of \cref{MonotoneNonHomotopic}\ref{edge:all}]
    Let $G$ be an $n$-vertex multigraph that has a monotone non-homotopic drawing in which every pair of edges cross at most $k$ times, where $k < n$ are integers. By \cref{drawingtosequences}, there exists a set $\AA$ of $\abs{E(G)}$ monotone drawing sequences of length $n$, each pair of which cross at most $k$ times. Applying \cref{sequencebound} to $\AA$ gives
    \begin{equation*}
        \abs{E(G)} = \abs{\AA} \leq 2 \cdot \binom{2n}{k + 1}. \qedhere
    \end{equation*}
\end{proof}

\section{Crossing Number Lemma for Monotone Drawings}\label{sec:crossing}

In this section we prove a tight lower bound on the number of crossings in monotone non-homotopic drawings, \cref{MonoNonHomCombinedCrossingLowerBound}, which we restate for the reader's convenience.

\MonoNonHomCombined*

The following definitions are due to \citet*{KPTU21}. A \defn{drawing style} is a collection of drawings of multigraphs. A drawing style $D$ is \defn{edge-deletion-closed} if for every drawing of a multigraph $G$ in style $D$, for every edge $e$ of $G$, the drawing of $G-e$ obtained from $\phi$ by deleting $e$ is in style $D$. 

Let $\phi$ be a drawing of a multigraph $G$, and let $v$ be a vertex of $G$. Let $B$ be a disc of positive radius centred at $\phi(v)$ containing no crossing points between edges in $G$. Consider the following operation: delete $v$ and the parts of the edges incident to $v$ within $B$, insert two vertices $v_1$ and $v_2$ in the interior of $B$, for each edge $e$ incident to $v$, join the end of $e$ at the boundary of $B$ to either $v_1$ or $v_2$, so that no edges cross within $B$. The resulting multigraph $G'$ and drawing $\phi'$ of $G'$ is a \defn{vertex split} of  $G$ at $v$. 

\begin{figure}[ht]
	\centering
	\includegraphics{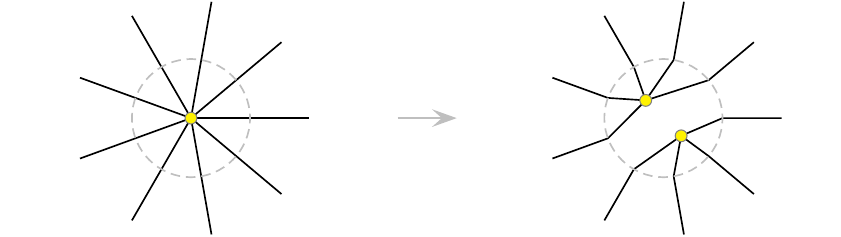}
	\caption{A vertex-split.}
\end{figure}

A drawing style $D$ is \defn{split-compatible} if for
every drawing of a multigraph $G$ in style $D$ and for every vertex split $G'$ of $G$, within the disc $B$ the edges can be embedded to obtain a drawing of $G'$ in style $D$. The \defn{planarisation} of a drawing $\phi$ of multigraph $G$ is the drawing $\phi'$ obtained from $\phi$ as follows: first make local adjustments to $\phi$ so that no three edges cross at a single point (while not changing the number of crossings), then insert a new vertex at each crossing point. We obtain a drawing $\phi'$ of a planar graph $G'$, in which each added vertex has degree 4.

The following lemma is implicitly proved by \citet*{KPTU21} (building on the work of \citet{PT20}, which employs the bisection-width method).

\begin{lemma}[\citep{KPTU21}]
\label{GeneralCrossingLemma}
Suppose $D$ is an edge-deletion-closed and split-compatible drawing style, and there are constants $c_1, c_2 > 0$ and $b > 1$ such that for every multigraph $G$:
\begin{itemize}
\item if $G$ has a drawing in style $D$ with no crossings, then $\abs{E(G)} \leq c_1\abs{V(G)}$;
\item if $G$ has a drawing in style $D$, then $\abs{E(G)} \leq c_2 \abs{V(G)}^b$.
\item if $G$ has a drawing in style $D$ with $C$ crossings, then $G$ has a drawing in style $D$ with at most $C$ crossings whose planarisation is separated. 
\end{itemize}
Then there exists a constant $\alpha=\alpha(b,c_1,c_2) > 0$ such that for
every multigraph $G$ with $\abs{E(G)}> (c_1+1)\abs{V(G)}$, 
every drawing of $G$ in style $D$ has at least $\alpha \abs{E(G)}^{2+1/(b-1)} / \abs{V(G)}^{1+1/(b-1)}$ crossings.
\end{lemma}

Let \defn{$M_k$}/\defn{$N_k$}/\defn{$O_k$} denote the collection of monotone non-homotopic multigraph drawings in which each pair of edges/incident edges/parallel edges cross at most $k$ times, respectively.

\begin{lemma}\label{Dk}
    $M_k$, $N_k$, and $O_k$ are edge-deletion-closed split-compatible drawing styles. Further if a multigraph $G$ has a drawing in style $M_k$/$N_k$/$O_k$ with $C$ crossings, then $G$ has a drawing in style $M_k$/$N_k$/$O_k$ with at most $C$ crossings whose planarisation is separated.
\end{lemma}
 
\begin{proof}
    (The proof of this result is very similar to the proof of an analogous result of \citet{KPTU21}.)\ 
    It is immediate that $M_k$, $N_k$, and $O_k$ are edge-deletion-closed. 
    We now show that they are all split-compatible. 
    Let $\phi$ be a drawing of a multigraph $G$ in style $M_k$, $N_k$, or $O_k$. Orient each edge of $G$ left-to-right. Let $v$ be a vertex of $G$. Let $E^-(v)$ and $E^+(v)$ be the sets of incoming and outgoing edges incident to $v$. Let $B$ be a disc of positive radius centred at $\phi(v)$ containing no crossing points between edges in $G$. As illustrated in \cref{SplitMonotone}, let $G'$ be a multigraph obtained by splitting $G$ at $v$ within disc $B$, where $v$ is replaced by $v_1$ and $v_2$. Consider each edge of $G'$ to inherit the orientation of the corresponding edge in $G$. Let $E_i$ be the set of edges of $G'$ incident to $v_i$. By assumption, no two edges in $E^-(v_1)\cup E^+(v_1)\cup E^+(v_2)\cup E^-(v_2)$ cross within $B$. Thus, the edges incident with $v_1$ or $v_2$ appear as $E^-(v_1),E^+(v_1),E^+(v_2),E^-(v_2)$ in the cyclic ordering determined by where the edges cross the boundary of $B$. Hence, the edges incident to $v_1$ or $v_2$ can be drawn within $B$ so that no two edges cross, each edge is x-monotone, and the resulting drawing of $G'$ is non-homotopic. Since each pair of parallel/incident edges in the new drawing were parallel/incident in the original, the styles $M_k$, $N_k$, and $O_k$ are all split-compatible.

    \begin{figure}[ht]
        \centering
        \includegraphics{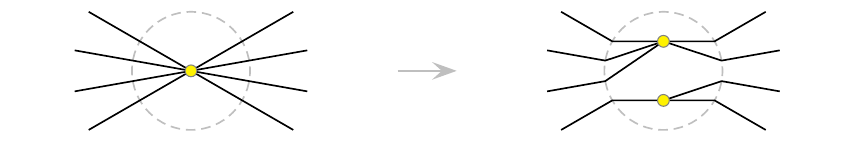}
        \caption{Splitting a vertex in a monotone drawing.}
        \label{SplitMonotone}
    \end{figure}

    Consider a drawing of a multigraph $G$ in style $M_k$/$N_k$/$O_k$ with $C$ crossings. By local adjustments, we may assume that no three edges cross at a single point (while maintaining monotonicity and not creating any crossings). We now show that $G$ has a drawing in the same style with at most $C$ crossings whose planarisation is separated. 
    
    Consider the drawing of $G$ in style $M_k$/$N_k$/$O_k$ with the smallest number of crossings. This minimal number is some $c \leq C$.
    Suppose there is a simple closed curve $\gamma$ formed by portions of two edges $e_1$ and $e_2$, where there is no vertex of $G$ in the interior of $\gamma$. Assume  that the interior of $\gamma$ is inclusion-minimal among all such curves. This implies that every time an edge enters $\gamma$ by crossing $e_1$, it must leave $\gamma$ by crossing $e_2$ and vice versa.
    
    As illustrated in \cref{RerouteMonotone}, redraw the diagram replacing $e_1$ with $e_1'$ which is the same as $e_1$ outside $\gamma$ and is $e_2 \cap \gamma$ on $\gamma$ and replacing $e_2$ with $e_2'$ which is the same as $e_2$ outside $\gamma$ and is very close to $e_1 \cap \gamma$ inside $\gamma$.
    \begin{figure}[ht]
        \centering
        \includegraphics{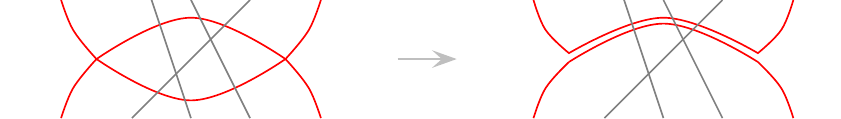}
        \caption{Rerouting a monotone drawing.}
        \label{RerouteMonotone}
    \end{figure}
    By the observation about edges entering and leaving $\gamma$, any crossing between an edge $e$ and $e_1'$ corresponds to a crossing between $e$ and $e_1$ and similarly for $e_2$. Thus the number of crossings between pairs of parallel edges/incident edges/edge in the resulting drawing is at most $k$. Furthermore, the resulting drawing has less than $c$ crossings (a crossing between $e_1$ and $e_2$ has been removed), is x-monotone, and is non-homotopic (since no vertex is in the interior of $\gamma$). 
    This contradiction shows that every simple closed curve $\gamma$ formed by portions of two edges of $G$ has a vertex of $G$ in its interior. A similar proof shows that $\gamma$ has a vertex of $G$ in its exterior. Therefore, the planarisation of this drawing is separated.
\end{proof}

\Cref{Dk} shows that \cref{GeneralCrossingLemma} is applicable to drawing style $M_k$/$N_k$/$O_k$. \Cref{MonotoneNonHomotopic}\ref{edge:all} with $k = 0$ shows that we may take $c_1 = 4$ for all three styles. \Cref{MonotoneNonHomotopic}\ref{edge:all} shows we may take $b = k + 1$ and $c_2 = \frac{2^{k + 2}}{(k + 1)!}$ for $M_k$. \Cref{MonotoneNonHomotopic}\ref{edge:incident} shows we may take $b = k + 2$ and $c_2 = \frac{2^{k + 1}}{(k + 1)!}$ for $N_k$. \Cref{MonotoneNonHomotopic}\ref{edge:parallel} shows we may take $b = k + 2$ and $c_2 = \frac{2^k}{(k + 1)!}$ for $O_k$.
All three parts of \cref{MonoNonHomCombinedCrossingLowerBound} immediately follow.

\section{Open Problems} \label{sec:conc}

Let $h(n, k)$ denote the maximum number of edges in an $n$-vertex multigraph that has a non-homotopic $k$-crossing drawing. \Cref{NonHomotopic} establishes the upper bound $h(n, k) \leqslant 2^{\OO(nk)}$. For lower bounds, \eqref{eq:PTTLB} gives $h(n, k) \geqslant 2^{\Omega(\sqrt{nk})}$ while the graphs $G_{n, k}$ from \cref{sec:MonotoneExample} give $h(n, k) \geqslant \sum_{\ell = 1}^k \binom{n}{\ell + 1}$. What is the true growth rate of $h(n, k)$?

In the monotone setting we have established the correct growth rate in terms of $n$. Let $h_\text{mon}(n, k)$ denote the maximum number of edges in a monotone non-homotopic $k$-crossing drawing of an $n$-vertex multigraph. \cref{MonotoneNonHomotopic}\ref{edge:all} and the graphs $G_{n, k}$ give $\sum_{\ell = 1}^k \binom{n}{\ell + 1} \leqslant h_\text{mon}(n, k) \leqslant 2 \cdot \binom{2n}{k + 1}$. Hence, $h_{\text{mon}} = \Theta_k(n^{k + 1})$ but the growth rate in terms of $k$ is unclear. It seems likely that the lower bound is closer to the truth. Is there some absolute constant $c$ such that $h_{\text{mon}}(n, k) \leqslant c \sum_{\ell = 0}^k \binom{n}{\ell + 1}$?

\subsection*{Acknowledgements} 
Many thanks to Torsten Ueckerdt for helpful conversations regarding \cref{GeneralCrossingLemma}. Thanks also to the referees for instructive feedback, especially for highlighting the results of \citet{Przytycki15}.

\fontsize{11pt}{12pt}
\selectfont
\def\soft#1{\leavevmode\setbox0=\hbox{h}\dimen7=\ht0\advance \dimen7 by-1ex\relax\if t#1\relax\rlap{\raise.6\dimen7 \hbox{\kern.3ex\char'47}}#1\relax\else\if T#1\relax \rlap{\raise.5\dimen7\hbox{\kern1.3ex\char'47}}#1\relax \else\if d#1\relax\rlap{\raise.5\dimen7\hbox{\kern.9ex \char'47}}#1\relax\else\if D#1\relax\rlap{\raise.5\dimen7 \hbox{\kern1.4ex\char'47}}#1\relax\else\if l#1\relax \rlap{\raise.5\dimen7\hbox{\kern.4ex\char'47}}#1\relax \else\if L#1\relax\rlap{\raise.5\dimen7\hbox{\kern.7ex \char'47}}#1\relax\else\message{accent \string\soft \space #1 not defined!}#1\relax\fi\fi\fi\fi\fi\fi}

\end{document}